\documentclass[11pt,reqno]{amsart}
\usepackage{latexsym}
\usepackage{amsmath}
\usepackage{amssymb, amscd, amsthm, mathtools}
\usepackage[all]{xy}
\usepackage{multirow}

\newtheorem{theorem}{{Theorem}}

\newtheorem{lemma}[theorem]{{Lemma}}

\theoremstyle{remark}
\newtheorem*{remark}{Remark}

\newcommand{\F}{\ensuremath{\mathbb F}}
\newcommand{\Z}{\ensuremath{\mathbb Z}}

\newcommand{\cO}{\mathcal{O}}
\newcommand{\Q}{\ensuremath{\mathbb Q}}
\newcommand{\cG}{\ensuremath{\mathcal G}}

\newcommand{\fp}{\mathfrak{p}}

\DeclareMathOperator{\End}{\mathrm{End}}

\DeclareMathOperator{\tr}{\mathrm{tr}}
\DeclareMathOperator{\Nrd}{\mathrm{Nrd}}

\DeclareMathOperator{\Trd}{\mathrm{Trd}}

\begin{document}

\title[Neighborhood  in the isogeny graph]{Neighborhood of  the supersingular elliptic curve isogeny graph  at $j=0$ and  $1728$}
\author{Songsong Li, Yi Ouyang, Zheng Xu}
\address{Wu Wen-Tsun Key Laboratory of Mathematics,  School of Mathematical Sciences, University of Science and Technology of China, Hefei, Anhui 230026, China}
\email{songsli@mail.ustc.edu.cn}
\email{yiouyang@ustc.edu.cn}
\email{xuzheng1@mail.ustc.edu.cn}

\date{}
 \maketitle

\begin{abstract}
We describe the neighborhood of the vertex $[E_0]$ (resp. $[E_{1728}]$) in the $\ell$-isogeny graph $\cG_\ell(\F_{p^2}, -2p)$ of supersingular elliptic curves over the finite field $\F_{p^2}$ when  $p>3\ell^2$ (resp. $p>4\ell^2$) with $E_0: y^2=x^3+1$ (resp. $E_{1728}: y^2=x^3+x$)  supersingular.
\end{abstract}

\section{Introduction and Main results}
Elliptic curves over finite fields play an important role in cryptography. It is well-known that elliptic curves defined over finite fields  can be classified into two types: ordinary and supersingular. If the elliptic curve $E$ is ordinary, the endomorphism ring of $E$ is an order of an imaginary quadratic field. If $E$ is supersingular, the endomorphism ring of $E$ is a maximal order of a quaternion algebra.
Computing the endomorphism rings and computing the isogenies of elliptic curves over finite fields are interesting problems in number theory and also has applications in cryptography. Stolbunov \cite{s} proposed a Diffie-Hellman type system based on the difficulty of computing isogenies between ordinary elliptic curves.  Cryptosystems based on the hardness of computing the endomorphism rings and isogenies of supersingular elliptic curves were proposed in \cite{jf}. Thus, it is important to find an explicit isogeny between two elliptic curves.

The efficient method to find explicit isogenies between supersingular elliptic curves is to use the isogeny graph, which is a Ramanujan graph introduced in \cite{cflm}. Childs, Jao and Soukharev gave an algorithm to compute ordinary elliptic curve isogenies in quantum subexponential time in \cite{cjs}. For supersingular elliptic curves defined over $\F_p$, from  \cite{dg, bjs}, there is also a subexponential time algorithm to solve this problem.

However, for supersingular elliptic curves defined over $\F_{p^2}$, it is hard to  compute the endomorphism rings or isogenies of the curves.  Let $\ell$ be a prime different from $p$. Here we recall the definition of
the isogeny graph $\cG_\ell(\F_{p^2})$ over $\F_{p^2}$ by Adj et al.~\cite{aam}. A vertex in the graph is an $\F_{p^2}$-isomorphism class $[E]$ of supersingular elliptic curves defined over $\F_{p^2}$. Let $[E_1]=[E_1']$, $[E_2]=[E_2']$ be two vertices in $\cG_\ell(\F_{p^2})$, let $\phi_1:\ E_1 \rightarrow E_2$ and $\phi_2: \ E'_1 \rightarrow E'_2$ be two $\ell$-degree $\F_{p^2}$-isogenies. We say that $\phi_1$ and $\phi_2$ are equivalent if there exist isomorphisms $\rho_1: E_1\rightarrow E_1'$ and $\rho_2: E_2\rightarrow E_2'$ such that $\phi_2\rho_1
=\rho_2 \phi_1$. Then an edge in the graph is an equivalent class of $\ell$-isogenies. If replacing the field of definition $\F_{p^2}$ of the curves and isogenies by the algebraic closure $\overline{\F}_p$ of $\F_p$, we get the definition of the isogeny graph $\cG_\ell(\overline{\F}_p)$. Note that for $E$  supersingular over $\F_{p^2}$, the trace of Frobenius $\pi=(x\mapsto x^{p^2})$ on the Tate module of $E$ must be $0$, $\pm p$ or $\pm 2p$.
For $t\in \{0,\pm p,\pm 2p\}$, let $\cG_\ell(\F_{p^2}, t)$ be the subgraph of $\cG_\ell(\F_{p^2})$ consisting of vertices $[E]$ with Frobenius trace $t$ and the adjacent edges.

Adj et al.\cite{aam} described clearly the subgraphs $\cG_\ell(\F_{p^2}, 0)$ and  $\cG_\ell(\F_{p^2}, \pm p)$. However, more work needs to be done when $t=\pm 2p$. Adj et al. proved the following key result in \cite[Theorem 6]{aam} and \cite [page10 line 24] {aam}:
 \begin{equation} \cG_\ell(\F_{p^2},2p) \cong \cG_\ell(\F_{p^2},-2p)\cong \cG_\ell(\overline{\F}_p).
 \end{equation}
Hence to study the neighborhood of a vertex $[E]$ in $\cG_\ell(\F_{p^2},\pm 2p)$, it suffices to study its neighborhood in $ \cG_\ell(\overline{\F}_p)$. Then tools such as Deuring's Correspondence Theorem can be used.

For $p>3$, there are two special supersingular elliptic curves over $\F_{p^2}$ with trace $-2p$:  \[ E_0: y^2=x^3+1 \ \text{when}\ p\equiv 2\bmod 3\]
with $j$ invariant $0$ and
 \[ E_{1728}: y^2=x^3+x \ \text{when}\ p\equiv 3\bmod 4 \]
with $j$-invariant $1728$. Then Adj et al.\cite[Theorems 10 and 12]{aam} and Ouyang-Xu \cite{OX} proved the following results about the loops on the vertices $[E_{1728}]$ and $[E_{0}]$ in the supersingular elliptic curves graph $\cG_\ell(\F_{p^2},-2p)$.
\begin{theorem} \label{theorem:loop} Suppose $\ell>3$.

$(1)$ If $p\equiv 3\bmod{4}$ and $p>4\ell$,  there are either $2$ or $0$ loops on $[E_{1728}]$ if $\ell\equiv 1\bmod{4}$ or $3\bmod 4$ respectively.

$(2)$ If $p\equiv 2\bmod{3}$ and $p>3\ell$, there are either $2$ or $0$ loops on $[E_{0}]$ if $\ell\equiv 1\bmod{3}$ or $2\bmod 3$ respectively.
\end{theorem}

In this paper, we shall work on the neighborhood of the vertices $[E_0]$ and $[E_{1728}]$ in $\cG_\ell(\F_{p^2},-2p)$. Our main result is

\begin{theorem}\label{theorem:main} Suppose $\ell>3$.

$(1)$ If $p\equiv 3\bmod{4}$ and $p>4\ell^2$,  there are $\frac{1}{2}\bigl(\ell-(-1)^{\frac{\ell-1}{2}}\bigr)$ vertices adjacent to  $[E_{1728}]$ in the graph, each connecting $[E_{1728}]$ with $2$ edges.
Moreover,  $1+(\frac{\ell}{p})$ of the vertices are of   $j$-invariants in $\F_p-\{1728\}$.

$(2)$ If $p\equiv 2\bmod{3}$ and $p>3\ell^2$, there are $\frac{1}{3}(\ell-(\frac{\ell}{3}))$  vertices adjacent to $[E_{0}]$ in the graph, each connecting $[E_{0}]$ with $3$ edges. Moreover,  $1+(\frac{-p}{\ell})$ of the vertices are of   $j$-invariants in $\F_p^*$.
\end{theorem}

\begin{remark}
(1) It would be best if  the bounds $4\ell^2$ and $3\ell^2$ can be improved to $4\ell$ and $3\ell$, as is the case for the number of loops in Theorem~\ref{theorem:loop}. However, this speculation is actually false. For a fixed prime $\ell>3$, let $P_1(\ell)$ (resp. $P_2(\ell)$) be the largest prime $p$ such that the number of vertices adjacent to $[E_{1728}]$ (resp. $[E_0]$) in $\cG_\ell(\F_{p^2},-2p)$ is smaller than  $\frac{1}{2}\bigl(\ell-(-1)^{\frac{\ell-1}{2}}\bigr)$ (resp. $\frac{1}{3}(\ell-(\frac{\ell}{3}))$), i.e., our main theorem fails for such a $p$. By numerical evidence presented in \S~\ref{sec:num}, for $5\leq \ell\leq 200$, most of the time $P_1(\ell)$ is the largest prime $\equiv 3\bmod{4}$ and smaller than $4\ell^2$, $P_2(47)=6599$ is the largest prime   $\equiv 2 \bmod 3$ and smaller than $3\times 47^2=6627$. In this sense, our bounds are sharp.

(2) For $\ell=2$ or $3$, we shall describe the neighborhood of $[E_0]$ and $[E_{1728}]$ in $\cG_\ell(\F_{p^2},-2p)$ for any prime $p>3$ (such that either $E_0$ or $E_{1728}$ is supersingular) in \S~\ref{sec:extra}.
\end{remark}
As the $j$-invariants of elliptic curves adjacent to $E_0$ (resp. $E_{1728}$) are roots of the modular polynomial $\Phi_\ell(0,X)$ (resp. $\Phi_\ell(1728,X)$), our result has the following immediate consequences about their roots.
\begin{theorem} \label{theorem:modular}
Suppose $\ell>3$.

$(1)$ If $p\equiv 3\bmod 4$ and $p>4\ell^2$, then if $\ell\equiv 3\bmod{4}$,
 \[ \Phi_\ell(1728,X)=\prod_{i=1}^{(\ell+1)/2} (X-a_i)^2 \]
with $1+(\frac{\ell}{p})$ of the roots $a_i\in \F_p-\{1728\}$ and the rest in $\F_{P^2}-\F_p$; if $\ell\equiv 1\bmod{4}$,
 \[ \Phi_\ell(1728,X)=(X-1728)^2 \prod_{i=1}^{(\ell-1)/2} (X-a_i)^2 \]
with $1+(\frac{\ell}{p})$ of the roots $a_i\in \F_p-\{1728\}$ and  the rest in $\F_{p^2}-\F_p$.

$(2)$ If $p\equiv 2\bmod 3$ and $p>3\ell^2$, then if $\ell\equiv 2\bmod{3}$,
 \[ \Phi_\ell(0,X)=\prod_{i=1}^{(\ell+1)/3} (X-a_i)^3 \]
with $1+(\frac{-p}{\ell})$ of the roots $a_i\in \F^*_p$ and the rest in $\F_{P^2}-\F_p$; if $\ell\equiv 1\bmod{3}$,
 \[ \Phi_\ell(0,X)=X^2 \prod_{i=1}^{(\ell-1)/3} (X-a_i)^3 \]
with $1+(\frac{-p}{\ell})$ of the roots $a_i\in \F^*_p$ and the rest in $\F_{p^2}-\F_p$.
\end{theorem}

\section{Preliminaries} \label{sec_formulas}
 \subsection{Elliptic curves over finite fields}
We recall basic facts about elliptic curves over finite fields. The general reference is \cite{js}. Let $\overline{\F}_p$ be the algebraic closure of $\F_p$.

An elliptic curves $E$ over the finite field $\F_q$ for $q$ a power of $p>3$ is defined by a Weierstrass equation $y^2=x^3+ax+b$ with $a,b\in \F_q$ and $4a^3+27b^2 \ne 0$.

The trace of the Frobenius $\pi: (x,y)\mapsto (x^q,y^q)$ on the Tate module of $E$, which we also call the trace of $E$ and denoted by $\tr(E)$, is the number $t=q+1-\# E(\F_q)$. The minimal polynomial of $\pi$ is $x^2-tx+q$ and Hasse's Theorem (see \cite{js}) implies that $|t|\leq 2\sqrt{q}$.

The $j$-invariant of $E$, which determines the isomorphism class of $E$ over $\overline{\F}_p$, is $j(E)=1728\cdot \frac{4a^3}{4a^3+27b^2}$.
The endomorphism ring $\End(E)$ of $E$ is the set of all isogenies between $E$ and itself. For $E$ an elliptic curve over $\F_q$, $\End(E)$ is either an order of an imaginary quadratic field, in which case $E$ is called ordinary; or a maximal order of a quaternion algebra, in which case  $E$ is called supersingular. It is well-known that  $E$ is ordinary (resp. supersingular) if and only if $p\nmid t$ (resp. $p\mid t$). Moreover, a supersingular elliptic curve $E$ over $\F_q$ always has $j$-invariant $j(E)\in \F_{p^2}$.

From now on, suppose $E$ is supersingular. Since $j(E)\in \F_{p^2}$, we assume $E$ is also defined over $\F_{p^2}$. Then $t=0$, $\pm p$ or $\pm 2p$.

 \subsection{Quaternion algebra}
A quaternion algebra over $\mathbb {Q}$ is of the form $H(a,b)=\Q+\Q i+\Q j+\Q k$, where $i^2=a$, $j^2=b$ and $k=ij=-ji$. The canonical involution on $H(a,b)$ is the map sending $\alpha=a_1+a_2i+a_3j+a_4k\in H(a,b)$ to $\bar {\alpha}=a_1-a_2i-a_3j-a_4k$.
The reduced trace of  $\alpha$ is $\Trd(\alpha)=\alpha+\bar{\alpha}=2a_1$, and the reduced norm is $\Nrd(\alpha)=\alpha\bar{\alpha}={a_1}^2-a{a_2}^2-b{a_3}^2+ab{a_4}^2$.
A subset $\Lambda$ is a lattice in $H(a,b)$ if $\Lambda=\Z x_1+\Z x_2+\Z x_3+\Z x_4$ and $\{x_1,x_2,x_3,x_4\}$ is a $\Q$-basis of $H(a,b)$.

The unique quaternion algebra over $\mathbb {Q}$ ramified only at $p$ and $\infty$ is $B_{p,\infty}=H(-1,-p)$.

\subsection{Orders and ideals in $B_{p,\infty}$}
An order $\cO$ of  $B_{p,\infty}$  is a subring of $B_{p,\infty}$ which is also a lattice, and is called a maximal order if it is not properly contained in any other order.

For $\cO$ an order of $B_{p,\infty}$, let $I$ be a left ideal of  $\cO$. The left order $\cO_L(I)$ and right order $\cO_R(I)$  of $I$ are defined to be
 \[ \cO_L(I)=\{x \in B_{p,\infty} \mid  xI \subseteq I\},\quad \cO_R(I)=\{x \in B_{p,\infty} \mid  Ix \subseteq I\}. \]
If $\cO$ is a maximal order, then $\cO_L(I)=\cO$ and $\cO_R(I)= \cO'$ is also a maximal order, in which case we say that $I$ connects $\cO$ and $\cO'$. Moreover, if $\cO$ is maximal,
 \[ \cO_R(I)= \cO\ \Longleftrightarrow\ I=\cO x\ \text{is principal}.\]
Define the reduced norm $\Nrd(I)$ of $I$ by
 \[ \Nrd(I)=\gcd(\{\Nrd(\alpha) \mid \alpha \in I \}). \]

\begin{lemma}\label{lemma:1} If $\cO$ is a maximal order in $B_{p,\infty}$ and $I$ is a left $\cO$-ideal of reduced norm $\ell$, then $\ell \in I$.
\end{lemma}

\begin{proof} Note that the abelian group $\cO/I$ is of order ${\Nrd(I)}^2=\ell^2$. Assume $\ell \notin I$. Then the image of $1$ in $\cO/I$ must be of order $\ell^2$ and   ${\cO}/I\cong \Z/\ell^2\Z$ is a cyclic group. Let $\{1,a,b,c\}$ be a $\Z$-basis of $\cO$. Let $t_a, t_b, t_c \in \Z$ such that $a-t_a$, $b-t_b$ and $c-t_c\in I$. We replace $\{a, b, c\}$ by $\{a-t_a, b-t_b, c-t_c\}$, then we get a $\Z$-basis $\{1,a,b,c\}$ of $\cO$ such that $\{\ell^2, a,b,c\}$ is a $\Z$-basis of $I$.  By computation, $\cO \subseteq \cO_R(I)$, since they are both maximal orders in $B_{p,\infty}$, $\cO_L(I)=\cO=\cO_R(I)$. From \cite[16.6.14]{v}, $\bar{I}I=\ell \cO_R(I)$, $I \bar{I}=\ell \cO_L(I)$. Thus $\bar{I}I=\ell \cO \subseteq \cO$, and $\bar{I} \subseteq \cO_L(I)=\cO$ by definition. Hence, $\ell \cO=\bar{I}I \subseteq \cO I \subseteq I$. We get a contradiction.
 \end{proof}

Now assume $\cO$ is a maximal order of $B_{p,\infty}$ containing a subring $\Z\langle i,j \rangle$ with $i^2=-q$, $j^2=-p$ and $ij=-ji$ such that $(q,p)=1$. Let $K=\Q(i)$. Then its ring of integers $\cO_K=\Z[i]$ if $q\equiv 1,\ 2\bmod{4}$ or $\Z[\frac{1+i}{2}]$ if $q\equiv 3\bmod{4}$. Let $R=\cO\bigcap K$. Then $\Z[i]\subseteq R\subseteq \cO_K$. Let $\epsilon=i$ if $R=\Z[i]$ or $\frac{1+i}{2}$ if $R=\cO_K=\Z[\frac{1+i}{2}]$.

Let $X_{\ell}$ be the set of of all left $\cO$-ideals of reduced norm $\ell$. Let $\hat{r}\in (R/\ell R)^\times$  and ${r}\in R$ a lifting of $\hat{r}$. Then for any $I\in X_\ell$, $\ell \cO+Ir$, depending only on $\hat{r}$ regardless the lifting,  is also in $X_\ell$ by using Lemma~\ref{lemma:1} and computing the reduced norm of elements in $\ell \cO+Ir$. Kohel et al ~\cite{klp} defined the action of $(R/{\ell R})^\times$ on $X_{\ell}$ by
\[(R/{\ell R})^\times \ \times \ X_{\ell} \ \rightarrow \ X_{\ell}; \ \ (\hat{r}, I) \ \mapsto \ \cO \ell+I r. \]
The following theorem was stated in \cite{klp} without a proof and we supply a proof here.
\begin{theorem} \label{theorem:21}
Assume $\cO$, $K$ and $R=\Z[\epsilon]$ as above.
Assume the prime $\ell \nmid 2pq[\cO:\Z\langle i,j\rangle]$.

$(1)$ If $\ell$ is inert in $R$, then $(R/{\ell R})^{\times}$  acts transitively on $X_\ell$ with $(\Z/{\ell \Z})^\times$ the stabilizer.

$(2)$ If $\ell$ splits in $R$,  write $\ell R=\mathfrak{p}_1\mathfrak{p}_2$, then $\{\cO\mathfrak{p}_1, \cO\mathfrak{p}_2\}\subset X_\ell$, $(R/{\ell R})^{\times}$  acts trivially on ${\cO\mathfrak{p}_1}$ and ${\cO\mathfrak{p}_2}$, and  acts transitively on $X_\ell-\{\cO\mathfrak{p}_1,\cO\mathfrak{p}_2\}$ with $(\Z/{\ell \Z})^\times$ the stabilizer.

In both cases, for any $I\in X_\ell-\{\cO\mathfrak{p}_1, \cO\mathfrak{p}_2\}$ and for $a\in \F_\ell$, let $I_a:= \ell\cO+ I (\tilde{a}+\epsilon)$ with $\tilde{a}\in \Z$ a lifting of $a$, then $X_\ell=\{I, I_a\mid a\in \F_\ell\}$.
\end{theorem}

\begin{proof}
As  $\ell \nmid pq [\cO:\Z\langle i,j\rangle]$,  $\cO/{\ell\cO}=\F_\ell \langle i,j\rangle$ with $i^2=-q$, $j^2=-p$ and $ij=-ji=k$. From ~\cite{gmm}, there is a ring isomorphism $\theta:\ \cO/{\ell\cO}\rightarrow M_2(\F_\ell)$ given by
\[1\mapsto \begin{pmatrix}                
    1 & 0\\
    0 & 1
\end{pmatrix},\ i\mapsto \begin{pmatrix}                
    0 & -q\\
    1 & 0
\end{pmatrix},\ j\mapsto \begin{pmatrix}                
    u & qv\\
    v & -u
\end{pmatrix}\]
where $(u,\ v)$ is a solution of $u^2+qv^2= -p$ in $\F_\ell$ (note that  this equation is always solvable  in $\F_\ell$).

From \cite[Theorem 6 in \S13]{ab}, since $M_2(\F_{\ell})$ is semi-simple and the only simple left $M_2(\F_{\ell})$-module is $2$-dimensional over $\F_\ell$, every nonzero proper left ideal must be a simple left $M_2(\F_{\ell})$-module and is generated by just one element $M\neq 0$. For $M$ is not invertible,  $\mathrm{rank}(M)=1$. Since $M_2(\F_\ell) M=M_2(\F_\ell) PM$ for any $P \in \mathrm{GL}_2(\F_{\ell})$, we may assume that $M=\omega:=\begin{pmatrix}0 & 0\\0 & 1\end{pmatrix}$ or $\omega_a:=\begin{pmatrix}1 & a\\0 & 0\end{pmatrix}$ for some $a \in \F_{\ell}$. Moreover, $\omega$ and $\omega_a$ generate different left ideals of $M_2(\F_\ell)$. Thus the set of all non-zero proper left ideals of $M_2(\F_{\ell})$ is the set
\[\overline{X_{\ell}}:=\{M_2(\F_\ell)\omega, M_2(\F_\ell)\omega_a\mid a\in \F_\ell\}. \]
Consequently, by the isomorphism $\theta$, $\cO/{\ell\cO}$ has $\ell+1$ non-zero proper left ideals, all of them are principal.

Under the canonical homomorphism $\cO\rightarrow \cO/{\ell\cO}$ and $\theta$, from Lemma~\ref{lemma:1}, the set $X_{\ell}$ maps bijectively to the set of nonzero left ideals of $\cO/{\ell\cO}$ and hence to $\overline{X_{\ell}}$. Moreover, the action of $(R/{\ell R})^{\times}$ on $X_{\ell}$ corresponds to the right multiplication of $(R/{\ell R})^{\times}$ on the set of nonzero left ideals of $\cO/{\ell\cO}$, and to the right multiplication action of  $\theta((R/{\ell R})^{\times})$  on $\overline{X_{\ell}}$.

Every element in $\theta((R/{\ell R})^{\times})$ is of the form $\begin{pmatrix}x & -qy\\y & x\end{pmatrix}$ with $x^2+qy^2 \neq 0$.  If $a_0\in \F_\ell$ satisfying $a_0^2+q = 0$, then $x+a_0y \neq 0$ and
 \[\omega_{a_0} \begin{pmatrix}x & -qy\\y & x\end{pmatrix}=\begin{pmatrix}x+a_0y & a_0(x+a_0y)\\0 & 0\end{pmatrix}\in M_2(\F_\ell)\omega_{a_0}. \]
 If $a\in \F_\ell$ satisfying $a^2+q \neq 0$, then for any $b$ such that $b^2+q\neq 0$,
 \begin{equation}~\label{eq:inert1}
 \omega_a \begin{pmatrix}\frac{q+ab}{q+a^2} & -q\frac{a-b}{q+a^2}\\\frac{a-b}{q+a^2} & \frac{q+ab}{q+a^2}\end{pmatrix}
 =\omega_b,
\end{equation}
 and
 \begin{equation}~\label{eq:inert2}
 \omega_a \begin{pmatrix}\frac{a}{q+a^2} & q\frac{1}{q+a^2}\\-\frac{1}{q+a^2} & \frac{a}{q+a^2}\end{pmatrix}=\begin{pmatrix}0 & 1\\0 & 0\end{pmatrix}\in M_2(\F_\ell)\omega.
 \end{equation}

Since $\ell\nmid 2q$, $\ell$ is prime to the conductor of $R$ and  $R\ell$ is the product of at most two prime $R$-ideals.

If $\ell$ is inert in $R$, there is no $a_0 \in \F_{\ell}$ such that ${a_0}^2+q=0$, thus the action of $\theta((R/\ell R)^{\times})$ on $\overline{X_{\ell}}$ is transitive by ~(\ref{eq:inert1}) and (\ref{eq:inert2}). In this case, $\{1, a+\epsilon\mid 0\leq a\leq \ell-1\}$ is a coset representative of $(\Z/\ell\Z)^\times$ in $(R/\ell R)^\times$, hence $X_\ell=\{I, I_a\mid a\in \F_\ell\}$ where $I$ is any element in $X_\ell$.

If $\ell R=\fp_1\fp_2$ splits in $R$, then $\fp_1=(\ell, a+\epsilon)$ and $\fp_2=\bar{\fp}_1=(\ell, a+\bar{\epsilon})$ for some $a\in \Z$ such that $N(a+\epsilon)=\Nrd(a+\epsilon)=0\in \F_\ell$, this implies that $\cO \fp_1=\ell \cO +\cO(a+\epsilon)$ and $\cO \fp_2=\ell \cO +\cO(a+\bar{\epsilon})$ are in $X_\ell$. This also implies that there exists some $a_0\in \F_\ell$ such that ${a_0}^2+q = 0$.
Thus  $\theta((R/\ell R)^{\times})$ has one orbit of length $\ell-1$ and two fixed points $\omega_{a_0}$, $\omega_{-a_0}$. In this case, $\{1, b+\epsilon\mid b\in \F_\ell, N(b+\epsilon)\neq 0\}$ is a coset representative of $(\Z/\ell\Z)^\times$ in $(R/\ell R)^\times$. For $I=\cO \fp_1=\ell \cO +\cO(a+\epsilon)$ or $\cO\fp_2$, $\ell\cO+Ir=\ell\cO+rI\subseteq I$, they must be equal since both are left $\cO$-ideals of reduced norm $\ell$. Thus for any $I\in X_\ell-\{ \cO\fp_1, \cO\fp_2\}$, we still have $X_\ell=\{I, I_a\mid a\in \F_\ell\}$.
\end{proof}

\begin{remark}
For any $I \in X_{\ell}$, from the proof of Theorem~\ref{theorem:21}, we have $I=\cO\ell+\cO\alpha$ for some $\alpha \in \cO$.
\end{remark}

From now on, by abuse of notation, we identify $\F_\ell$ with the set $\{0,\cdots,\ell-1\}$ and $\tilde{a}$ with $a$ in the definition of $I_a$.

\subsection{Supersingular elliptic curves and $B_{p,\infty}$}

Suppose $E$ is a supersingular elliptic curve over $\F_{p^2}$, then
$\End(E)=\cO$ is a maximal order of a quaternion algebra $B_{p,\infty}$. For $I$  a left integral ideal of $\cO$, let $E[I]$=\{$P \in {E} \mid {\alpha}(P)=O$ for every $\alpha \in{I}$\}, then the isogeny
 \[ {\phi}_I:E \rightarrow {E}_I =E/E[I] \]
has $\ker\phi_I=E[I]$ and  $\deg(\phi_I)=\Nrd(I)$ the reduced norm of $I$. On the other hand, if $\phi: E\rightarrow E'$ is an isogeny of degree $n$, then $\ker\phi$ is of order $n$ and $I_\phi=\{\alpha\in \cO\mid \alpha(P)=O\ \text{for all}\ P\in \ker\phi\}$ is a left $\cO$-ideal of reduced norm $n$.
Deuring's Correspondence Theorem (see Voight \cite[chapter 42]{v} or \cite{d}) is the following theorem:

\begin{theorem} \label{theorem:22}
Let $E$ be a supersingular elliptic curve over $\F_{p^2}$ and $\End(E)=\cO$. Then $\cO$ is a maximal order of $B_{p,\infty}$.

$(1)$ There is a $1$-to-$1$ correspondence between left ideals $I$ of $\cO$ of reduced norm $n$ and equivalent classes of isogenies $\phi: E\rightarrow E'$ of degree $n$ given by $I\mapsto [\phi_I]$ and $[\phi]\mapsto I_\phi$.

$(2)$ If $\phi:E \rightarrow E'$ and $I$ are corresponding to each other, then  $\End(E')\cong \mathcal {O}_R(I)$ is the right order of $I$ in $B_{p,\infty}$. In particular, $\phi\in \End(E)$ if and only if $I=I_\phi=\cO \phi$ is principal.

$(3)$ Suppose that $\phi_1:E \rightarrow E_1$, $\phi_2:E \rightarrow E_2$ are two isogenies corresponding to the left ideals $I_1,I_2\subseteq \cO$. Then  $E_1$ and $E_2$ are in the same isomorphism class if and only if $I_1={I_2}x$ for some $x\in B_{p,\infty}$. i.e.  $I_1$ and $I_2$ are in the same left ideal class.
\end{theorem}

\section{Proof of Main Theorem}
By the isomorphism $\cG_\ell(\F_{p^2},-2p)\cong \cG_\ell(\overline{\F}_p)$, to study the neighborhoods of $[E_{0}]$ and $[E_{1728}]$ in $\cG_\ell(\F_{p^2},-2p)$, it suffices to study the neighborhoods of $[E_{0}]$ and $[E_{1728}]$ in $\cG_\ell(\overline{\F}_p)$.

From \cite{m}, in the cases when $E_0$ or $E_{1728}$ is supersingular ($p\equiv 2\bmod 3$ for $E_0$ and $p\equiv 3\bmod 4$ for $E_{1728}$), then
 \begin{equation}
  \End(E_{0}) = \Z+\Z\frac{1+i}{2}+ \Z j+\Z\frac{3+i+3j+k}{6} \end{equation}
where $i^2=-3,\ j^2=-p$ and $ij=-ji=k$; and
\begin{equation} \End(E_{1728}) =\Z+\Z i+ \Z\frac{1+j}{2}+\Z\frac{i+k}{2}\end{equation}
where $i^2=-1,\ j^2=-p$ and  $ij=-ji=k$.

We shall apply Theorem~\ref{theorem:21} in both cases. Let $I$ be a left $\End(E_0)$ or $\End(E_{1728})$ ideal of reduced norm $\ell$ not above $\ell$, then the set $X_\ell$ of all left ideals  of reduced norm $\ell$ is $\{I,\ I_a\mid a\in \F_\ell\}$ by Theorem~\ref{theorem:21}. The strategy of our proof is to find all left ideal classes of $X_\ell$ and the size of each ideal class. Then applying Deuring's Theorem, we obtain  information of vertices and edges in the neighborhoods of $[E_0]$ and $[E_{1728}]$.

We need the following easy lemma:
\begin{lemma} \label{lemma:fact} Let $N$ be a $\Z$-module and $M$ a submodule of $N$. Then for coprime integers $n$ and $m$, $mM+nN=M+nN$.
\end{lemma}

\subsection{Neighborhood of $[E_{1728}]$}
This subsection is  devoted to the proof of Theorem~\ref{theorem:main}(1).

Write $\cO=\End(E_{1728})=\Z+\Z i+ \Z\frac{1+j}{2}+\Z\frac{i+k}{2}$
 where $i^2=-1,\ j^2=-p$ and $ij=-ji=k$.  In this case $R=\cO\cap\Q(i)=\Z[i]=\cO_{\Q(i)}$  is a principal ideal domain and its unit group is $\{\pm 1,\pm i\}$.

\begin{lemma} \label{lemma:ell1} Suppose $\ell\equiv 1\bmod{4}$.

$(1)$ $\ell$ splits completely in $\Z[i]$,  $\ell\Z[i]= (m+ni)\Z[i]\cdot (m-ni)\Z[i]$ with  $(m,n)\in \Z^2$ being any solution of $X^2+Y^2=\ell$. The solution set of $X^2+Y^2=\ell$ is $\{(\pm m,\pm n), (\pm n,\pm m)\}$.

$(2)$ The set of pairs $(x,y)\in \Z^2$ satisfying $\ell\nmid x$ and $X^2+Y^2=\ell^2$ is $\{(\pm (m^2-n^2),\pm 2mn), (\pm 2mn, \pm (m^2-n^2))\}$.

$(3)$ The two left $\cO$-ideals $\cO (m+ni)$ and $\cO(m-ni)$ are of reduced norm $\ell$. Moreover, for $J$ any left $\cO$-ideal of reduced norm $\ell$, let $b=m/n\in \F_\ell$, then $b^2=-1$, $J_b=\ell\cO+ J (\tilde{b}+i)=\ell\cO+ J(m+ni)=\cO(m+ni)$ and $J_{-b}=\cO(m-ni)$.
\end{lemma}
\begin{proof} All except the last part of (3) are classical results in number theory. That $b^2=-1$ is clear. By Lemma~\ref{lemma:fact}, $J_b=\ell\cO+J(m+ni)\subseteq \cO(m+ni)$,  but both of them are left $\cO$-ideals of reduced norm $\ell$, we have $J_b=\cO (m+ni)$. Similarly $J_{-b}=\cO(m-ni)$.
\end{proof}

\begin{lemma} \label{lemma:ell2} Suppose $p>4\ell^2$.

$(1)$ If  $\mu\in \ell^{-1}\cO$, $\Nrd(\mu)=1$ and $\mu\notin\{\pm 1,\pm i\}$,  then $\ell\equiv 1\bmod{4}$ and $\mu=\ell^{-1}(x+yi)$ where $(x,y)\in \Z^2$ satisfies $\ell\nmid x$ and $X^2+Y^2=\ell^2$.

$(2)$ If $s\in \ell^{-1}\cO$, $s^2=-p$ and $s\notin \{\pm j,\pm k\}$,
then $\ell\equiv 1\bmod{4}$, and $s=\ell^{-1}(xj+yk)$  where $(x,y)\in \Z^2$ satisfies $\ell\nmid x$ and $X^2+Y^2=\ell^2$.
\end{lemma}
\begin{proof} (1)  Write
 \[ \mu=\frac{1}{\ell}\Bigl(A+Bi+C\frac{1+j}{2}+D\frac{i+k}{2}\Bigr),\quad  (A,B,C,D\in \Z). \]
By the fact $\Nrd(\mu)=1$, then
 \[ \Bigl(A+\frac{C}{2}\Bigr)^2+\Bigl(B+\frac{D}{2}\Bigr)^2
 +\frac{p(C^2+D^2)}{4}=\ell^2. \]
If $p>4\ell^2$, then $C=D=0$ and hence $\mu=\frac{A+Bi}{\ell}$ and $A^2+B^2=\ell^2$. If $\ell\mid A$, then $\ell\mid B$ and $\mu\in \{\pm 1,\pm i\}$. If $\ell\nmid A$, then $(B/A)^2=-1\in \F_\ell$ and $\ell\equiv 1\bmod{4}$.

(2) Write $s={\ell^{-1}}(a+bi+c\frac{1+j}{2}+d\frac{i+k}{2})$. Then $s^2=-p\in \Q$ implies  $a+\frac{c}{2}=0$ and $c\in 2\Z$. Moreover,
 \[ \ell^2 p=-\ell^2 s^2=\Nrd(\ell s)=  \left(b+\frac{d}{2}\right)^2+\frac{p}{4}(c^2+d^2) \]
implies $p\mid 2b+d$. If  $2b+d\neq 0$, then $(b+\frac{d}{2})^2 \geq \frac{p^2}{4} >p\ell^2$ since $p>4\ell^2$, impossible. Hence $b+\frac{d}{2}=0$ and $d\in 2\Z$. Hence $s \in  \ell^{-1}\cO$ with $s^2=-p$ must have the form $s=\ell^{-1}(xj+yk)$ with $x,y\in \Z$ and $x^2+y^2=\ell^2$. If $\ell\mid x$, then $\ell\mid y$. It means $s=Xj+Yk$($X,Y\in \Z$), for $\Nrd(Xj+Yk)=(X^2+Y^2)p$, then the square roots of $-p$ in $\Z[j,k]$ are $\{\pm j, \pm k\}$. And we get $s\in \{\pm j, \pm k\}$. Otherwise we again have $(y/x)^2=-1\in \F_\ell$ and $\ell\equiv 1\bmod{4}$.
\end{proof}

 \begin{proof}[Proof of Theorem~\ref{theorem:main}(1)]
 Let $\ell \ne p$ be a prime. If $\ell\equiv 1\bmod{4}$, let $(m,n)\in \Z^2$ be any solution of $x^2+y^2=\ell$ and $b=m n^{-1}\in \F_\ell$ in this case, then $b^2=-1$.

Let $I$ be a left $\cO$-ideal of reduced norm $\ell$ different from $\cO(m \pm ni)$ if $\ell \equiv 1 \bmod 4$. By Theorem~\ref{theorem:21}, the set of the $\ell+1$ left $\cO$-ideals of reduced norm $\ell$ is
 \[ X_\ell=\{I,\ I_a=\ell\cO+I(a+i)\mid a\in \F_\ell=\{0,\cdots, \ell-1\} \}. \]
If $\ell\equiv 1\bmod{4}$, by Lemma~\ref{lemma:ell1}, $I_b=\cO (m+ni)$ and $I_{-b}=\cO(m-ni)$.

We claim that  $I i=I_0$, $I_0 i=I$ and $I_a i=I_{-a^{-1}}$ if $a\neq 0$. Indeed, from Lemma~\ref{lemma:1}, $\ell \in I$. Then $\ell i\in I$ and $\ell=-\ell i i\in I i$, hence $I_0=I i+\ell \cO=Ii$ and $I=I_0 i$. For $a\neq 0$ in $\F_\ell$, $i \in \cO$ then $\ell \cO i \subseteq \ell \cO$, and $I_ai$ is left-$\cO$ ideal, $\ell \cO \subseteq I_ai$. It implies $I_ai=\ell\cO i+I(-1+ai)=\ell\cO+\ell\cO i+I(-a^{-1}+i)=\ell\cO+I(-a^{-1}+i)=I_{-a^{-1}}$, where the second identity is by Lemma~\ref{lemma:fact}.

To summarize, we  divide $X_\ell$ into $\frac{\ell+1}{2}$ subsets, each consisting of $2$  elements in the same ideal class: $\{I, I_0\}$, $\{I_a, I_{-a^{-1}}\}$ $(a^2\neq 0, -1)$ and $\{I_b, I_{-b}\}$ for $b^2=-1$. We show that any two left ideals in different subsets are not in the same ideal class by contradiction.

Suppose $I$ and $J$ are from different subsets of $X_\ell$ and $I=J\mu$ for some $\mu \in B_{p,\infty}$, then $\mu\notin \{\pm 1,\pm i\}$ and $\Nrd(\mu)=1$.
Since $\ell \in J$, $\ell \mu \in I\subseteq \cO$ and $\mu\in \ell^{-1}\cO$. By Lemma~\ref{lemma:ell2}(1), we have $\ell\equiv 1\bmod{4}$,  $\mu=\frac{A+Bi}{\ell}$, $A^2+B^2=\ell^2$ and $\ell\nmid A$. This means that $A+Bi=u(m\pm ni)^2$ for $u\in \{\pm 1,\pm i\}$, and thus $\gcd(A+Bi,\ell)$ in $\Z[i]$ is  $u  (m\pm ni)$. In particular $ I_{\pm b}=\cO(m+ni)\subseteq I$ and hence $I_{\pm b}=I$ as both are of the same reduced norm.  Switch the role of $I$ and $J$, we  get $J=I_{\pm b}$. Hence both $I$ and $J$ are in the same subset $\{I_b, I_{-b}\}$, impossible. By Deuring's Theorem (Theorem~\ref{theorem:22}), and from Theorem~\ref{theorem:loop}, when $\ell \equiv 3 \bmod 4$, none of the subsets consist of
ideals corresponding to endomorphisms. This completes the proof of the first part of  Theorem~\ref{theorem:main}(1).

For the second part, let $E$ be a supersingular elliptic curve defined over $\F_p$ such that $E_{1728}$ connects to $E$ via a left $\cO$-ideal of reduced norm $\ell$. By \cite[Proposition 2.4]{dg}, a supersingular elliptic curve is defined over $\F_p$ if and only if $\Z[\sqrt{-p}]$ is contained in its endomorphism ring. Then $\End(E)=\cO_R(I)\subseteq \ell^{-1}{\cO}$ has an element $s$ such that $s^2=-p$. By Lemma~\ref{lemma:ell2}(2),  we know either $s\in \{\pm j, \pm k\}$ or in the case $\ell\equiv 1\bmod{4}$, $\ell s=x j+yk$, $(x,y)\in \Z^2$ such that $\ell\nmid x$ and $x^2+y^2=\ell^2$.

Let $\hat{\cO}=\cO/\ell\cO$. Then $\hat{\cO}$ is a quaternion algebra over $\F_\ell$. We can identify $\hat{\cO}$ with $M_2(\F_\ell)$ via the isomorphism $\theta$ in Theorem~\ref{theorem:21} with $q=1$.  Moreover, the set $\overline{X_\ell}$ defined in Theorem~\ref{theorem:21}
corresponds to $X_\ell$ bijectively.
Let  $I_a$ be the  left $\cO$-ideal of reduced norm $\ell$ corresponding to   $\hat{I_a}=M_2(\F_\ell)\omega_a$ in $\overline X_\ell$.
For $s\in \cO$, let $\hat{s}$ be the image of $s$ in $M_2(\F_\ell)$. By abuse of notation, write $i$, $j$, $k$ for $\hat{i}$, $\hat{j}$ and $\hat{k}$.

If $(\frac{-p}{\ell})=1$, let $t\in \F_\ell$  such that $t^2= -p$ and let $(u,v)=(t,0)$. In this case, $I_\infty=\cO\ell+\cO(-t+j)$, $I_a=\cO\ell+\cO(-t+j)(a+i)$. Then one can easily check that $\hat{I}_\infty j\subset \hat{I}_\infty$,  $\hat{I}_0 j\subset \hat{I}_0$ and $\hat{I}_a j\nsubseteq \hat{I}_a$ for all other $a$, this means $j\in \cO_R(I_\infty)$, $j\in \cO_R(I_0)$ but $j\notin \cO_R(I_a)$ for other $a$. Similarly $k\in \cO_R(I_{\pm 1})$ and $k\notin \cO_R(I_a)$ for other $a$. Now if $\ell\equiv 1\bmod{4}$ and $(x,y)$ any solution that $x^2+y^2=\ell^2$ and $\ell\nmid x$, then one can check $\hat{I}_a (\hat{x}j+\hat{y}k)\neq 0$ if $a^2\neq -1$, hence $\cO (xj+yk)\nsubseteq \ell \cO$ and $\ell^{-1}(xj+yk)\notin O_R(I_a)$ if $a^2\neq -1$. If $a^2=-1$, then $I_a=\cO(m+ni)$ or $I_a=(m-ni)$ for $m^2+n^2=\ell$, corresponding to the loops. In conclusion, there are two vertices defined over $\F_p$  adjacent to $[E_{1728}]$, one corresponding to the ideal class $[I_\infty]=[I_{0}]$ and the other corresponding to the ideal class $[I_{1}]=[I_{-1}]$.

If $(\frac{-p}{\ell})=-1$, then $uv\neq 0$ for any solution $(u,v)$ of $X^2+Y^2=-p$. It is easy to check $\omega j\notin \hat{I}_\infty$. For $a\in \F_\ell$, $\omega_a j \in \hat{I}_a$ implies that $2au=(1-a^2) v$. From $uv\neq 0$, then $a\neq 0,\pm 1$ and $v=\frac{2a}{1-a^2} u$. Hence $-p=\frac{(1+a^2)^2}{(1-a^2)^2} u^2$, impossible. This means $j\notin I_a$ for all $a\in \F_\ell\cup \{\infty\}$. Similarly $k\notin I_a$ for all $a\in \F_\ell\cup \{\infty\}$.  Also, if $x,y\ne0$ such that $\omega_a (x+yi)j =0 $, by computation, $a^2+1=0$, which corresponds to the loops. In conclusion, there is no  vertex defined over $\F_p$ other than $[E_{1728}]$.
\end{proof}

\subsection{Neighborhood of $[E_{0}]$}
This subsection is devoted to the proof of Theorem~\ref{theorem:main}(2).

 Write $\cO=\End(E_{0})=\Z+\Z \frac{1+i}{2}+ \Z\frac{i+k}{3}+\Z\frac{j+k}{2}$
 where $i^2=-3,\ j^2=-p$ and $ij=-ji=k$. Write $\epsilon=\frac{1+i}{2}$. Then
 $\Z[\epsilon]=\cO\cap \Q(i)$, as the ring of integers of $\Q(i)$, is a principal ideal domain and its unit group is $\{\pm 1,\pm \epsilon, \pm \bar{\epsilon} \}$.

\begin{lemma} \label{lemma:ell3} Suppose $\ell\equiv 1\bmod{3}$.

$(1)$ $\ell$ splits completely in $\Z[\epsilon]$,  $\ell\Z[\epsilon]= (m+n\epsilon)\Z[\epsilon]\cdot (m+n\bar{\epsilon})\Z[\epsilon]$ with  $(m,n)\in \Z^2$ being any solution of $X^2+XY+Y^2=\ell$. The solution set of $X^2+XY+Y^2=\ell$ is $\{\pm(m,n),\pm(n,m), \pm (m+n,-n), \pm (m+n,-m), \pm(-n,m+n), \pm (-m,m+n)  \}$.

$(2)$ The set of pairs $(x,y)\in \Z^2$ satisfying $\ell\nmid x$ and $X^2+XY+Y^2=\ell^2$ is  $\{\pm( (m^2-n^2), n^2+2mn), (\pm (n^2-m^2), \pm (m^2+2mn)), \pm( (m^2+2mn), -(n^2+2mn)), \pm( (n^2+2mn), -(m^2+2mn)), \pm( (m^2+2mn), n^2-m^2), \pm( (n^2+2mn), m^2-n^2) \}$.

$(3)$ The two left $\cO$-ideals $\cO (m+n\epsilon)$ and $\cO(m+n\bar{\epsilon})$ are of reduced norm $\ell$. Moreover, for $J$ any left $\cO$-ideal of reduced norm $\ell$, let $b=m/n\in \F_\ell$, then $b^2+b+1=0$, $J_b=\ell\cO+ J (\tilde{b}+\epsilon)=\ell\cO+ J(m+n\epsilon)=\cO(m+n\epsilon)$ and $J_{b^2}=\cO(m+n\bar{\epsilon})$.
\end{lemma}
\begin{proof} Similar to the proof of Lemma~\ref{lemma:ell1}.
\end{proof}

\begin{lemma} \label{lemma:ell4} Suppose $p>3\ell^2$.

$(1)$ If  $\mu\in \ell^{-1}\cO$, $\Nrd(\mu)=1$ and $\mu\notin\{\pm 1,\pm \epsilon, \pm  \bar{\epsilon} \}$,  then $\ell\equiv 1\bmod{3}$ and $\mu=\ell^{-1}(x+y\epsilon)$ where $(x,y)\in \Z^2$ satisfies $\ell\nmid x$ and $X^2+XY+Y^2=\ell^2$.

$(2)$ If $s\in \ell^{-1}\cO$, $s^2=-p$ and $s\notin \{\pm j,\pm \epsilon j, \pm  \bar{\epsilon}j \}$,
then $\ell\equiv 1\bmod{3}$, and $s=\ell^{-1}(x+y\epsilon)j$  where $(x,y)\in \Z^2$ satisfies $\ell\nmid x$ and $X^2+XY+Y^2=\ell^2$.
\end{lemma}
\begin{proof} (1)  Write
 \[ \mu=\frac{1}{\ell}\Bigl(A+B\frac{1+i}{2}+C\frac{i+k}{3}+D\frac{j+k}{2}\Bigr),\quad  (A,B,C,D\in \Z). \]
By the fact $\Nrd(\mu)=1$, then
 \[ \Bigl(A+\frac{B}{2} \Bigr)^2+ 3\Bigl(\frac{B}{2}+\frac{C}{3}\Bigr)^2
 + \frac{p}{3}(C^2+3CD+3D^2)=\ell^2. \]
If $p>3\ell^2$, then $C=D=0$ and hence $\mu=\frac{A+B\frac{1+i}{2}}{\ell}$ and $A^2+AB+B^2=\ell^2$. If $\ell\mid A$, then $\ell\mid B$ and $\mu\in \{\pm 1,\pm \frac{1+i}{2}, \pm\frac{1-i}{2} \}$. If $\ell\nmid A$, then $(B/A)^2+(B/A)+1=0 \in \F_\ell$ and $\ell\equiv 1\bmod{3}$.

(2) Write $s={\ell^{-1}}(a+b\frac{1+i}{2}+c\frac{i+k}{3}+d\frac{j+k}{2})$. Then $s^2=-p\in \Q$ implies  $a+\frac{b}{2}=0$ and $b\in 2\Z$. Moreover,
 \[ \ell^2 p=-\ell^2 s^2=\Nrd(\ell s)=  3\left(\frac{b}{2}+\frac{c}{3}\right)^2+\frac{p}{3}(c^2+3cd+3d^2) \]
implies $p\mid \frac{3}{2}b+c$. If  $\frac{3}{2}b+c\neq 0$, then $3(\frac{b}{2}+\frac{c}{3})^2 \geq \frac{p^2}{3} >p\ell^2$ since $p>3\ell^2$, impossible. Hence $\frac{3}{2}b+c=0$ and $c \in 3\Z$. Hence $s \in  \ell^{-1}\cO$ with $s^2=-p$ must have the form $s=\ell^{-1}(x+y\frac{1+i}{2})j$ with $x,y\in \Z$ and $x^2+xy+y^2=\ell^2$. If $\ell\mid x$, then $\ell\mid y$ It means $s=(X+Y\frac{1+i}{2})j$($X,Y\in \Z$), for $\Nrd((X+Y\frac{1+i}{2})j)=(X^2+XY+Y^2)p$, then the square roots of $-p$ in $\Z[j,\frac{1+i}{2}j]$ are $\{\pm j, \pm \frac{1+i}{2}j, \pm \frac{1-i}{2}j \}$. And we get $s\in \{\pm j, \pm \frac{1+i}{2}j, \pm \frac{1-i}{2}j \}$. Otherwise we again have $(y/x)^2+(y/x)+1=0\in \F_\ell$ and $\ell\equiv 1\bmod{3}$.
\end{proof}

\begin{proof}[Proof of Theorem~\ref{theorem:main}(2)]
 Let $\ell \ne p$ be a prime. If $\ell\equiv 1\bmod{3}$, let $(m,n)\in \Z^2$ be any solution of $x^2+xy+y^2=\ell$ and $b=m n^{-1}\in \F_\ell$ in this case, then $b^2+b+1=0$.

Let $I$ be a left $\cO$-ideal of reduced norm $\ell$ different from $\cO(m+n\epsilon)$ and $\cO(m+n\bar{\epsilon}) $ if $\ell \equiv 1 \bmod 3$. By Theorem~\ref{theorem:21}, the set of the $\ell+1$ left $\cO$-ideals of reduced norm $\ell$ is
 \[ X_\ell=\{I,\ I_a=\ell\cO+I(a+\frac{1+i}{2})\mid a\in \F_\ell=\{0,\cdots, \ell-1\} \}. \]
If $\ell\equiv 1\bmod{3}$, by Lemma~\ref{lemma:ell3}, $I_b=\cO (m+n\epsilon)$ and $I_{b^2}=\cO(m+n \bar{\epsilon})$.

We claim that  $I \epsilon=I_0$, $I  \bar{\epsilon}=I_{-1}$ and $I_a \epsilon =I_{-(a+1)^{-1}}$, $I_a  \bar{\epsilon}=I_{-a^{-1}(a+1)}$ if $a\neq 0,-1$. Indeed, from Lemma~\ref{lemma:1}, $\ell \in I$. Then $\ell , \ell \bar{\epsilon} \in I$ and $\ell=\ell  \bar{\epsilon} \epsilon \in I \epsilon$, hence $I_0=I \epsilon+\ell \cO=I\epsilon$. Similarly, $I \bar{\epsilon}=I_{-1}$. For $a\neq 0,-1$ in $\F_{\ell}$, $\epsilon \in \cO$ then $\ell \cO \epsilon \subseteq \ell \cO$, and $I_a\epsilon$ is left-$\cO$ ideal, $\ell \cO \subseteq I_a\epsilon$(likely $\ell \cO \subseteq I_a\bar{\epsilon}$) . Then
 \[ \begin{split} I_a\epsilon &=\ell\cO \epsilon +I(-1+(a+1)\epsilon)= \ell\cO+\ell\cO \epsilon +I(-1+(a+1)\epsilon) \\ &=\ell\cO+I(-1+(a+1)\epsilon) =\ell\cO+I(-(a+1)^{-1}+\epsilon)\\
 &=I_{-(a+1)^{-1}}, \end{split}\]
 \[ \begin{split} I_a\bar{\epsilon} &=\ell\cO\bar{\epsilon}+I((a+1)-a\epsilon)=\ell\cO+\ell\cO\bar{\epsilon}+I((a+1)-a\epsilon)\\
 &=\ell\cO+I((a+1)-a\epsilon)=
 \ell\cO+
 I(-(a+1)a^{-1}+\epsilon)\\
 &=I_{-a^{-1}(a+1)}, \end{split}\] where the second identity is by Lemma~\ref{lemma:fact}.

To summarize, we  divide $X_\ell$ into  $[\frac{\ell+2}{3}]$ subsets, each consisting of $2$ or $3$ elements in the same ideal class: $\{I, I_0, I_{-1} \}$, $\{ I_a, I_{-(a+1)^{-1}}, I_{-a^{-1}(a+1)} \}$  $(a^2+a+1 \neq 0, 1)$ and $\{I_b, I_{b^2}\}$ for $b^2+b+1=0$. We show that any two left ideals in different subsets are not in the same ideal class by contradiction.

Suppose $I$ and $J$ are from different subsets of $X_\ell$ and $I=J\mu$ for some $\mu \in B_{p,\infty}$, then $\mu\notin \{\pm 1, \pm \epsilon, \pm \bar{\epsilon} \}$ and $\Nrd(\mu)=1$.
Since $\ell \in J$, $\ell \mu \in I\subseteq \cO$, and $\mu\in \ell^{-1}\cO$. By Lemma~\ref{lemma:ell4}(1), we have $\ell\equiv 1\bmod{3}$,  $\mu=\frac{A+B\epsilon}{\ell}$, $A^2+AB+B^2=\ell^2$ and $\ell\nmid A$. This means that $A+B\epsilon=u(m+n\frac{1\pm i}{2})^2$ for $u\in \{\pm 1,\pm \epsilon, \pm \bar{\epsilon} \}$, and thus $\gcd(A+B\epsilon,\ell)$ in $\Z[\frac{1+\sqrt{-3}}{2}]$ is  $u  (m+n\frac{1\pm i}{2})$. In particular $ I_{b}=\cO(m+n\epsilon)\subseteq I$ or $ I_{b^2}=\cO(m+n\bar{\epsilon})\subseteq I$  and hence $I_{b}=I$ or $I_{b^2}=I$ as both are of the same reduced norm.  Switch the role of $I$ and $J$, we  get $J=I_{b}$ or $J=I_{b^2}$. Hence both $I$ and $J$ are in the same subset $\{I_b, I_{b^2}\}$, impossible. By Deuring's Theorem (Theorem~\ref{theorem:22}), and from Theorem~\ref{theorem:loop}, when $\ell \equiv 2 \bmod 3$, none of the subsets consist of
ideals corresponding to endomorphisms. This completes the proof of the first part of  Theorem~\ref{theorem:main}(2).

For the second part, let $E$ be a supersingular elliptic curve defined over $\F_p$ such that $E_{0}$ connects to $E$ via a left $\cO$-ideal of reduced norm $\ell$. By \cite[Proposition 2.4]{dg}, a supersingular elliptic curve is defined over $\F_p$ if and only if $\Z[\sqrt{-p}]$ is contained in its endomorphism ring. Then $\End(E)=\cO_R(I)\subseteq \ell^{-1}{\cO}$ has an element $s$ such that $s^2=-p$. By Lemma~\ref{lemma:ell4}(2),  we know either $s\in \{\pm j, \pm \epsilon j, \pm \bar{\epsilon} j\}$ or in the case $\ell \equiv 1 \bmod 3$, $\ell s=x j+y \epsilon j$, $(x,y)\in \Z^2$ such that $\ell\nmid x$ and $x^2++xy+y^2=\ell^2$.

Let $\hat{\cO}=\cO/\ell\cO$. Then $\hat{\cO}$ is a quaternion algebra over $\F_\ell$. We can identify $\hat{\cO}$ with $M_2(\F_\ell)$ via the isomorphism $\theta$ in Theorem~\ref{theorem:21} with $q=3$. Moreover, the set \[ \overline{X}_\ell=\{\hat{I}_\infty= M_2(\F_\ell) \begin{pmatrix} 0 &0  \\ 0 & 1 \end{pmatrix},\ \hat{I}_a :=M_2(\F_\ell) \begin{pmatrix} 1 & 2a+1  \\ 0 & 0 \end{pmatrix}\ (a\in \F_\ell)\} \]
corresponds to $X_\ell$ bijectively. For our convenience, the form of the set $\overline{X_\ell}$ here is different from that in the proof of Theorem~\ref{theorem:21}.
Let  $I_a$ be the left $\cO$-ideal of reduced norm $\ell$ corresponding to  $\hat{I}_a$.
For $s\in \cO$, let $\hat{s}$ be the image of $s$ in $M_2(\F_\ell)$. By abuse of notation, write $i$, $j$, $k$ for $\hat{i}$, $\hat{j}$ and $\hat{k}$.

If $(\frac{-p}{\ell})=1$, let $t\in \F_\ell$  such that $t^2= -p$ and let $(u,v)=(t,0)$. In this case, $I_\infty=\cO\ell+\cO(-t+j)$, $I_a=\cO\ell+\cO(-t+j)(a+\epsilon)$. Then one can easily check that $\hat{I}_\infty j\subset \hat{I}_\infty$,  $\hat{I}_{-2^{-1}} j\subset \hat{I}_{-2^{-1}}$ and $\hat{I}_a j\nsubseteq \hat{I}_a$ for all other $a$, this means $j\in \cO_R(I_\infty)$, $j\in \cO_R(I_{-2^{-1}})$ but $j\notin \cO_R(I_a)$ for other $a$. Similarly $\epsilon j\in \cO_R(I_{0}), \cO_R(I_{-2})$ and $\epsilon j \notin \cO_R(I_a)$ for other $a$. Also, $\bar{\epsilon} j\in \cO_R(I_{-1}), \cO_R(I_{1})$ and $\bar{\epsilon} j \notin \cO_R(I_a)$ for other $a$. Now if $\ell\equiv 1\bmod{3}$ and $(x,y)$ any solution that $x^2+xy+y^2=\ell^2$ and $\ell\nmid x$, then one can check $\hat{I}_a (\hat{x}j+\hat{y}\epsilon j)\neq 0$ if $a^2+a+1 \ne 0$, hence $\cO (xj+y\epsilon j )\nsubseteq \ell \cO$ and $\ell^{-1}(xj+y\epsilon j)\notin O_R(I_a)$ if $a^2+a+1 \ne 0$. If $a^2+a+1=0$, then $I_a=\cO(m+n\epsilon)$ or $I_a=(m+n\bar{\epsilon})$ for $m^2+mn+n^2=\ell$, corresponding to the loops. In conclusion, there are two vertices defined over $\F_p$  adjacent to $[E_{0}]$, one corresponding to the ideal class $[I_\infty]=[I_{0}]=[I_{-1}]$ and the other corresponding to the ideal class $[I_{1}]=[I_{-2^{-1}}]=[I_{-2}]$.

If $(\frac{-p}{\ell})=-1$, then $uv\neq 0$ for any solution $(u,v)$ of $X^2+3Y^2=-p$. It is easy to check $\begin{psmallmatrix} 0 &0  \\ 0 & 1 \end{psmallmatrix} j\notin \hat{I}_\infty$. For $a\in \F_\ell$, $\begin{psmallmatrix} 1 &a  \\ 0 & 0 \end{psmallmatrix} j \in \hat{I}_a$ implies that $2(2a+1)u=(3-(2a+1)^2) v$. From $v\neq 0$, then $2a+1\neq 0$ and $u=\frac{3-(2a+1)^2}{2(2a+1)} v$. Hence $-p=\frac{(3+(2a+1)^2)^2}{(2(2a+1))^2} v^2$, impossible. This means $j\notin I_a$ for all $a\in \F_\ell\cup \{\infty\}$. Similarly $\epsilon j, \bar{\epsilon} j \notin I_a$ for all $a\in \F_\ell\cup \{\infty\}$.  Also, if $x,y\ne0$ such that $\begin{pmatrix} 1 & 2a+1  \\ 0 & 0 \end{pmatrix} \begin{pmatrix} 2x+y & -3y  \\ y & 2x+y \end{pmatrix} \begin{pmatrix} u & 3v  \\ v & -u \end{pmatrix} =0 $, by computation, $a^2+a+1=0$, which corresponds to the loops. In conclusion, there is no  vertex defined over $\F_p$ other than $[E_{0}]$.
\end{proof}

\section{Numerical Evidence} \label{sec:num}

For a fixed prime $\ell>3$, let $P_1(\ell)$ (resp. $P_2(\ell)$) be the largest prime $p$ such that the number of vertices adjacent to $[E_{1728}]$ (resp. $[E_0]$) in $\cG_\ell(\F_{p^2},-2p)$ is smaller than  $\frac{1}{2}\bigl(\ell-(-1)^{\frac{\ell-1}{2}}\bigr)$ (resp. $\frac{1}{3}(\ell-(\frac{\ell}{3}))$), i.e., our main theorem fails for such a $p$.  Let $P'_1(\ell)$ (resp. $P'_2(\ell)$) be the largest prime $p$ such that  $p\equiv 3 \bmod 4$  and  $p<4\ell^2$  (resp.  $p\equiv 2 \bmod 3$  and $p<3\ell^2$).
By Theorem~\ref{theorem:main},  $P_i(\ell)\leq P'_i(\ell)$. The equality $P_1(\ell)= P'_1(\ell)$ (resp. $P_2(\ell)= P'_2(\ell)$) holds only when our bound $4\ell^2$ (resp. $3\ell^2$) is sharp, in this case  we say Bound I (resp. Bound II) is satisfied for $\ell$.

We compute the  values of $P_{1}(\ell)$ and $P_{2}(\ell)$ for  $5\le \ell \leq 200$ and list them in Table 1.

\begin{table}[!htbp]
\caption{The values of $P_1(\ell)$ and $P_2(\ell)$ for  $5\le \ell \le 200$}
\tiny
 \begin{tabular}{rrrrrrrrrrrrrrrrrrrrr|}
     \hline
     \hline
       $\ell $ & 5 & 7 & 11 & 13  & 17 & 19 & 23 & 29 & 31 & 37 & 41  \\
       \hline \hline
       $P_1(\ell)$   & 83 & 191 & 479 & 659 & 1151 & 1439 &2111 &3359 & 3803 & 5471 & 6719  \\ \hline
       $P_2(\ell)$  & 47& 71 & 311 & 479 & 839 & 1031 & 1559 & 2447 & 2711 & 4079 & 4967    \\ \hline
       Bound  &  I &  I &  I  & $\times$ &  I  &  I &  I  &  I  & $\times$ &  I  & I   \\ \hline
 \hline \hline

$\ell $ & 43 & 47 & 53 & 59 & 61  & 67 & 71 & 73 & 79 & 83 & 89   \\
        \hline \hline
       $P_(\ell)$ & 7351 & 8831 & 11171 & 13907 & 14879  &17939 & 20147 & 21227 & 24923 &27551 &31667   \\ \hline
       $P_2(\ell)$  & 5519 &6599 & 8231 & 10391 & 11087  & 13259 & 14951 & 15959 & 18671 & 20639 &23687   \\ \hline
       Bound  &  I &  I,II &  I  & I &  I  &  I &  I  &  $\times$  & $\times$ &   I   & I  \\ \hline
\hline\hline
       $\ell $ & 97 & 101  & 103 & 107 & 109 & 113 & 127 & 131 & 137 & 139 & 149  \\
       \hline \hline
       $P_1(\ell)$  & 37619 & 40787 & 42407 & 45779 & 47507 & 51071 & 64499 &68639 & 75011 &77279 &88799  \\ \hline
       $P_2(\ell)$ & 28151 & 30491 &31799 &34319 & 35591 & 38231 & 48311 &51431 &56099 &57839 &66491 \\ \hline
       Bound  &   I  &   I  &   I   &  I  &   I   &   I  &  I   &  I   &  I  &   I   &  I   \\ \hline
       \hline\hline

     $\ell $ &151 &157&163&167&173&179 &181 & 191 & 193 & 197 &199 \\
\hline \hline
       $P_1(\ell)$ &91199 & 98543 &106187 &111539 & 119699 & 128159 & 130927 &145879 & 148991 & 155231 & 158363 \\ \hline
       $P_2(\ell)$   & 68351 & 73823 & 79631 & 83639 & 89759 & 95819 & 98207 & 109391 & 111623 & 116351 &118751  \\ \hline
       Bound  &   I  &  $\times$ &  $\times$  &  I  &  I   &  I  &  $\times$  &  $\times$  &  I  &   I   &  I   \\ \hline

\end{tabular}
\end{table}

As can be seen from Table 1, of the $44$ primes between $5$ and $200$, Bound I is satisfied for $36$  primes. The prime $47$ is the only $\ell<200$ satisfying Bound II (and also Bound I), but the difference  $P'_2(\ell)-P_2(\ell)$ for each $\ell$ is not  big. In this sense  our bounds are sharp.

\section{The cases when $\ell =2$ and $3$} \label{sec:extra}
For completeness, we list results here for the cases $\ell=2$ or $3$.

(1) For the curve $E_{1728}$ (hence $p\equiv 3\bmod{4}$),
\begin{itemize}
 \item[$\ell=2$] If $p>4\ell=8$, then $[E_{1728}]$ has $1$ loop by \cite[Theorem 10]{aam}, and if $I$ is non-principal of reduced norm $\ell$, then $[I]=[I_0]$,  $[E_{1728}]$ connects to another vertex by $2$ edges; if $p=7$, then $\Phi_2(X,1728)\equiv (X-1728)^3 \bmod 7$,  $[E_{1728}]$ has $3$ loops.

 \item[$\ell=3$] If $p>4\ell=12$, then $E_{1728}$ has no loop, and the $4$ edges correspond to $2$ ideal classes, $[E_{1728}]$ connects to $2$ other vertices by $2$ edges each; if $p=7$, then $\Phi_3(X,1728)\equiv (X+1)^4 \bmod 7$, $E_{1728}$ connects to another vertex by $4$ edges; if $p=11$, then $\Phi_3(X,1728)\equiv (X^2+X+10)^2 \bmod 11$, which means $E_{1728}$ connects to $2$ vertices by $2$ edges each.
\end{itemize}

(2) For the curve $E_{0}$ (hence $p\equiv 2\bmod{3}$),
\begin{itemize}
 \item[$\ell=2$] If $p>3\ell=6$, then $E_{0}$ has no loop by \cite{OX} and $I,I_0,I_{-1}$ are in the same ideal class, which means $[E_{0}]$ connects to another vertex by $3$ edges; if $p=5$, then $\Phi_2(X,0)\equiv X^3 \bmod 5$, which means $[E_{0}]$ has $3$ loops.

 \item[$\ell=3$] If $p>3\ell=9$, then $[E_{0}]$ has $1$ loop by \cite{OX} and $I,I_0,I_{-1}$ are in the same ideal class, which means $[E_{0}]$ connects to another vertex by $3$ edges; if $p=5$, then $\Phi_3(X,0)\equiv X^4 \bmod 5$, which means $[E_{0}]$ has $4$ loops.
\end{itemize}

\subsection*{Acknowledgement}
Research is partially supported by Anhui Initiative in Quantum Information Technologies (Grant No. AHY150200) and NSFC (Grant No. 11571328).

\end{document}